\documentclass[review,11 pt]{amsart}

\usepackage{lineno,hyperref}
\usepackage{amsmath,amssymb}
\usepackage{lineno}
\newtheorem{thm}{Theorem}[section]
\newtheorem{defn}[thm]{Definition}
\newtheorem{prop}[thm]{Proposition}

\newtheorem{lem}[thm]{Lemma}
\newtheorem{rem}[thm]{Remark}

\newtheorem{ex}[thm]{Example}

\begin{document}


\title{Characterizations of multiframelets on $\mathbb{Q}_{p}$}

\author[D. Haldar]{Debasis Haldar}

\address{Department of Mathematics\\ NIT Rourkela\\ Rourkela 769008\\ India}

\email{debamath.haldar@gmail.com}
\author[A. Bhandari]{Animesh Bhandari}
\address{Dept. of Mathematics\\ NIT Meghalaya\\ Shillong 793003\\ India}
\email{bhandari.animesh@gmail.com}



\begin{abstract}
This paper presents a discussion on  $p$-adic multiframe by means of
its wavelet structure, called as multiframelet,  which is build upon $p$-adic wavelet construction. Multiframelets create much excitement in mathematicians as well as engineers on account of its tremendous potentiality to analyze rapidly changing transient signals. Moreover, multiframelets can produce more accurately localized temporal and frequency information, due to this fact it  produce a methodology to reconstruct signals by means of decomposition technique. Various properties of multiframelet sequence in $L^{2}(\mathbb{Q}_{p})$ have been analyzed. Furthermore, multiframelet set in $\mathbb{Q}_{p}$
has been engendered and scrutinized.
\end{abstract}

\maketitle

\section{Introduction}
In 1952, the notion of Hilbert space Frames was introduced by Duffin and Schaeffer in \cite{1} on the study of  nonharmonic Fourier series. Basically frames are extension of orthonormal bases in Hilbert spaces. Frames allows every element of the corresponding Hilbert space to be written as linear combination of frame elements using the associated frame coefficients,  where the coefficients are not unique and due to this reason frames are sometimes called overcomplete system. After several decades in 1986, the importance of frames was realized by  Daubechies, Grossmann and Meyer, took the key step of connecting frames with wavelets and Gabor systems in \cite{3}. Further Young re-introduce the same in his book \cite{2}, which contains basic facts about  frame and Grochenig has given the nontrivial extension of frames to Banach spaces in \cite{8}. Frames having wavelet structures, have been popularized through several generalizations and significant applications, for detail discussion regarding the same readers are  referred to \cite{Ak10, Bh18, 9, 10, 3, 1, 12, 2}.

The field $\mathbb{Q}_{p}$ of $p$-adic numbers is defined as the completion of $\mathbb{Q}$ with respect to metric topology induced by the $p$-adic norm $|\cdot|_{p}$. The $p$-adic norm is defined as follows :
	$$|x|_{p}= \begin{cases}
	0 & \text{ if } x=0, \\
	p^{- \gamma} & \text{ if } x=p^{\gamma}\frac{m}{n} \neq 0 ,~~ p \not| ~ mn.
	\end{cases}$$

This norm has the ultrametric property leading to the strong triangle inequality $\lvert x+y \rvert_{p} \leq \text{max} \{\lvert x \rvert_{p}, \lvert y \rvert _{p} \}$. Here the equality holds if and only if $|x|_{p} \neq |y|_{p}$. Thus the $p$-adic norm is non-Archimedean. Every non-zero $p$-adic number admits Laurent series expansion in $p$ given by
$$x= \sum \limits_{j=\gamma}^{\infty} x_{j}p^j,$$
where $x_{j} \in \{ 0,1,...., p-1 \}$ with $x_{\gamma} \neq 0$ and $\gamma \in \mathbb{Z}$. The fractional part of $x$ is denoted as $\{x\}_{p}$ and is defined by $ \sum \limits_{j=\gamma}^{-1}x_{j}p^j$. Thus  $\{x\}_{p} =0$ if and only if $\gamma \geq 0$. Further $\{0\}_{p}:=0$. The ring of $p$-adic integers is denoted as $\mathbb{Z}_{p}$ and the set of fractional numbers $I_{p}$ are given by
\begin{center}
	$\mathbb{Z}_{p} = \{ x \in \mathbb{Q}_{p} : \{x\}_{p}=0 \}$ and $I_{p} = \{ x \in \mathbb{Q}_{p}: \{x\}_{p} = x \}$.
\end{center}  
Therefore, $I_{p}=\{ \frac{a_{-\gamma}}{p^{\gamma}} + \frac{a_{-\gamma +1}}{p^{\gamma -1}} + \ldots + \frac{a_{-1}}{p} \in \mathbb{Q}_{p} : 0 \leq a_{i} \leq p-1, i \in \mathbb{Z}, a_{-\gamma} \neq 0 \}$. The additive character $\chi_{p}$ on the field $\mathbb{Q}_{p}$ is defined by 
\begin{center}
	$\chi_{p}(x)= e^{2\pi i \{x\}_{p}}$,  $ x \in \mathbb{Q}_{p}$.
\end{center}
 The field $\mathbb{Q}_{p}$ is locally compact, totally disconnected and has no isolated points. There exists Haar measure $dx$ on $\mathbb{Q}_p$ which is positive and invariant under translations, i.e., $d(x+a)=dx$, for all $a \in I_{p}$. It is normalized by the equality $\int \limits_{\mathbb{Z}_{p} }dx =1$. Moreover,
$$d(ax)= |a|_{p}~ dx,~~~  a \in \mathbb{Q}_{p} \setminus \{0 \}.$$

The Hilbert space of all complex-valued functions on $\mathbb{Q}_{p}$, square integrable with respect to the measure $dx$, is denoted by $ L^{2}(\mathbb{Q}_{p}) $. The inner product in this space is given by
\begin{center}
	$\langle f,g \rangle = \int \limits_{\mathbb{Q}_{p}} f(x) \overline{g(x)}dx,$ where $f, g \in L^{2}(\mathbb{Q}_{p}).$
\end{center}

Throughout the paper we denote $\mathcal I$ as the identity operator, $\mathfrak{L}=\{1,2, \cdots, L \}$, $\mathbf{1}_{X}$ as characteristic function of $X$ and $\mathcal{R}_{U}$ is denoted as range of bounded linear operator $U$.\\ 

The exposition of the article is as follows, Section 2 presents basic discussions regarding multiframelets. Furthermore, characterizations of multiframelets through various $p$-adic settings have been analyzed in Section 3.

 \section{Preliminaries and Background}
 Before divining into the main results, throughout this section we discuss some fundamental definitions and preliminary results regarding multiframelets that aid us to produce various characterizations of the same.
\begin{defn} (Multiframelet)
    A set of functions $\mathcal{F}$=$\{f^{(1)}, \ldots, f^{(L)}\} \subset L^{2}(\mathbb{Q}_{p})$ is said to be a multiframelet of order L if $\{f^{(l)}_{j,a}:=p^{\frac{j}{2}}f^{(l)}(p^{-j} \cdot -a) : j \in \mathbb{Z},~ a \in I_{p},~ l \in \mathfrak L\}$ is a frame for $L^{2}(\mathbb{Q}_{p})$ i.e.  there exist $ A,B > 0$ so that for all $g \in L^{2}(\mathbb{Q}_{p})$ we have,
    \begin{equation} \label{dmf}
    A\left \| g \right \|^{2} \leq \sum \limits_{ l \in \mathfrak{L}} \sum \limits_{j \in \mathbb{Z}} \sum \limits_{a \in I_{p}}| \langle g,f^{(l)}_{j,a} \rangle |^{2} \leq B\left \| g \right \|^{2}.
    \end{equation}
\end{defn}
When $L=1$, $\mathcal{F}$ is simply said to be a framelet. $A$ and $B$ are said to be lower and upper multiframelet bounds respectively. Clearly, they are not unique. The optimal lower multiframelet bound is the supremum of all lower multiframelet bounds and the optimal upper multiframelet bound is the infimum of all upper multiframelet bounds. This number $B$ is sometimes called Besselet bound as the second inequality is called as Besselet inequality and corresponding $\mathcal{F}$ is called
$\textbf{Besselet}$. An arbitrary Besselet need not imply existance of lower multiframelet bound and hence not a multiframelet. A multiframelet which ceases to be a multiframelet on the removal of any one of its vectors is termed an $\textbf{exact multiframelet}$. $\mathcal{F}$ is said to be a $\textbf{tight multiframelet}$ if it is possible to choose $A = B$ and $\mathcal{F}$ is said to be a $\textbf{normalized tight multiframelet}$ or $\textbf{Parseval
multiframelet}$ if it is possible to choose $A = B = 1$. Every
orthonormal basis is a Parseval multiframelet but a Parseval
multiframelet need not be orthogonal or a basis.

\begin{ex}
    Kozyrev's multiwavelet (cf. \cite{5}) given by
    \begin{equation} \label{kw}
    \theta_{k}(x) = \chi_{p}(p^{-1}kx) \mathbf{1}_{\mathbb{Z}_p}(x) ,  \hspace{1 cm}  x \in \mathbb{Q}_{p},
    \end{equation}
    where $k = 1, 2,\ldots, p-1$.
\end{ex}

\begin{ex}
    Khrennikov and Shelkovich's multiwavelet (cf. \cite{4}) given by
    \begin{equation}
    \theta ^{(m)}_{s}(x)=\chi_{p}(sx) \mathbf{1}_{\mathbb{Z}_p}(x) , \hspace{.3 cm} x \in \mathbb{Q}_{p}, \nonumber
    \end{equation}
    where $s \in J_{p,m}:= \{ \frac{s_{-m}}{p^m} + \ldots +\frac{s_{-1}}{p} :~ s_{-j} = 0, 1,\ldots, p-1$; $j = 1, 2,\ldots, m;~ s_{-m} \neq 0 \}$, $m \in \mathbb{N}$ is fixed.
\end{ex}

    


 Let $\{f_{j,~a}^{(l)} : j \in \mathbb{Z},~ a \in I_{p},~ l \in \mathfrak{L}\}$ be a multiframelet, then the corresponding synthesis operator $T : \ell^{2}(\mathfrak{L} \times \mathbb{Z} \times I_{p}) \rightarrow L^{2}(\mathbb{Q}_{p})$ is defined as $T \{c(l,j,a)\} = \sum \limits_{l \in \mathfrak L} \sum \limits_{j \in \mathbb{Z}} \sum \limits_{a \in I_{p}} c(l,j,a) f^{(l)}_{j,a}$ and the adjoint operator of $T$, $T^{*} :  L^{2}(\mathbb{Q}_{p}) \rightarrow \ell^{2}(\mathfrak{L} \times \mathbb{Z} \times I_{p})$, is given by $T^* g = \{ \langle g,f_{j,a}^{(l)} \rangle :~ j \in \mathbb{Z},~ a \in I_{p},~ l \in \mathfrak{L}\}$, is called as analysis operator. By composing synthesis and analysis operator, we obtain the associated multiframelet operator $\mathcal{S} :L^{2}(\mathbb{Q}_{p}) \rightarrow L^{2}(\mathbb{Q}_{p})$, defined as $\mathcal{S}g=TT^{*}g= \sum \limits_{l \in \mathfrak L} \sum \limits_{j \in \mathbb{Z}} \sum \limits_{a \in I_{p}} \langle g,f^{(l)}_{j,a} \rangle f^{(l)}_{j,a}$ and therefore equation (\ref{dmf}) can be written as $A\left \| g \right \|^{2} \leq \langle \mathcal{S}g, g \rangle \leq B\left \| g \right \|^{2}$ for all $g \in L^{2}(\mathbb{Q}_{p})$. 

\begin{rem}
It is to be noted that $\mathcal{S}$ is  bounded, linear, self-adjoint,  bijective operator, which is evident from the following discussion:

 Clearly $\mathcal{S}$ is well-defined and hence linearity follows from the linearity property of inner product.
	
	\noindent \underline{Injectivity}
	$\mathcal{S}g =0 \Rightarrow \left \|g \right \|=0$ (by frame condition) $\Rightarrow g=0$. Hence $\mathcal{S}$ is injective.
	
	\noindent \underline{Surjectivity} Let $g \in (\text{Im} ~ \mathcal{S})^{\bot}$. Then $\langle \mathcal{S}f,g \rangle =0, ~~ \forall ~ f \in L^{2}(\mathbb{Q}_{p})$.\\
	In particular, $\langle \mathcal{S}g , g \rangle =0 \Rightarrow \left \|g \right \| =0 \Rightarrow g=0$. So $\text{Im}~ \mathcal{S} =L^{2}(\mathbb{Q}_{p})$. Thus $\mathcal{S}$ is surjective and hence bijective. \\
	Positivity and boundedness of $\mathcal{S}$ directly follows from the multiframelet condition.
	
	\noindent \underline{Self-adjointness} For $f, g \in
	L^{2}(\mathbb{Q}_{p})$, let's consider $\langle \mathcal{S}f, g	\rangle = \langle f, \mathcal{S}^{*}g \rangle$ in order to calculate
	$\mathcal{S}^*$. Now
	\begin{eqnarray}
	\langle \mathcal{S}f, g \rangle = \left \langle \sum
	\limits_{l \in \mathfrak L} \sum \limits_{j \in \mathbb{Z}} \sum \limits_{a
		\in I_{p}} \langle f,f^{(l)}_{j,a} \rangle f^{(l)}_{j,a} , g
	\right \rangle \nonumber
	=& \sum \limits_{l \in \mathfrak L} \sum \limits_{j \in \mathbb{Z}} \sum \limits_{a \in I_{p}} \langle f,f^{(l)}_{j,a} \rangle \langle f^{(l)}_{j,a} ,g  \rangle \nonumber \\
	=& \sum \limits_{l \in \mathfrak L} \sum \limits_{j \in \mathbb{Z}} \sum \limits_{a \in I_{p}} \langle f, \langle g,f^{(l)}_{j,a} \rangle f^{(l)}_{j,a} ~\rangle \nonumber \\
	=& \left \langle f, \sum \limits_{l \in \mathfrak L} \sum \limits_{j \in
		\mathbb{Z}} \sum \limits_{a \in I_{p}} \langle g,f^{(l)}_{j,a}
	\rangle f^{(l)}_{j,a} \right \rangle \nonumber
	\end{eqnarray}
	So $\mathcal{S}^*g = \sum \limits_{l \in \mathfrak L} \sum \limits_{j \in
		\mathbb{Z}} \sum \limits_{a \in I_{p}} \langle g, f^{(l)}_{j,a}
	\rangle f^{(l)}_{j,a} = \mathcal{S}g$. Thus $\mathcal{S}^* =
	\mathcal{S}$ and hence $\mathcal{S}$ is self-adjoint.
\end{rem}

\begin{defn}(Dual Multiframelet)
	Let $\{f_{j,~a}^{(l)} : j \in \mathbb{Z},~ a \in I_{p},~ l \in \mathfrak{L} \}$ and $\{g_{j,~a}^{(l)} : j \in \mathbb{Z},~ a \in I_{p},~ l \in \mathfrak{L} \}$ be multiframelet for $L^{2}(\mathbb{Q}_{p})$. Then $\{g_{j,~a}^{(l)} : j \in \mathbb{Z},~ a \in I_{p},~ l \in \mathfrak{L} \}$ is said to be a dual multiframelet of $\{f_{j,~a}^{(l)} : j \in \mathbb{Z},~ a \in I_{p},~ l \in \mathfrak{L} \}$ if  every $h \in L^{2}(\mathbb{Q}_{p})$ can be written as $h= \sum \limits_{l \in \mathfrak{L}} \sum \limits_{j \in \mathbb{Z}} \sum \limits_{a \in I_{p}} \langle h ,g_{j,a}^{(l)} \rangle  f_{j,a}^{(l)}$.
\end{defn}

\begin{defn}(Multiframelet Sequence) A sequence $\{f_{j,~a}^{(l)} : j \in \mathbb{Z},~ a \in I_{p},~ l \in \mathfrak{L}\}$ is said to be a multiframelet sequence if it is a multiframelet for $\overline{\text{span}} \{f_{j,~a}^{(l)} : j \in \mathbb{Z},~ a \in I_{p},~ l \in \mathfrak{L} \}$.
\end{defn}

\begin{rem}
It is to be concluded every multiframelet is a multiframelet sequence.
\end{rem}

\begin{center} \section{Main Results} \end{center}
In this section our primary intention  is to produce various characterizations of multiframelets.

Every frame sequence satisfies  Bessel's  inequlity, for detail discussion regarding the same we refer  \cite{BhMu19}. Analogous result is also satisfied for multiframelet.

\begin{prop}
	Every multiframelet sequence is a Besselet.
\end{prop}

\begin{proof}
	Let $\{f_{j,~a}^{(l)} : j \in \mathbb{Z},~ a \in I_{p},~ l \in \mathfrak{L}\}$ be a multiframelet sequence in
	$L^{2}(\mathbb{Q}_{p})$. So it is a multiframelet for $H=\overline{span}\{f_{j,~a}^{(l)} : j \in \mathbb{Z},~ a \in I_{p},~ l \in \mathfrak{L}\}$. Since $L^{2}(\mathbb{Q}_{p}) = H \oplus
	H^\perp$, then for every $g \in L^{2}(\mathbb{Q}_{p})$,  $g = g_{H} + g_{H^\perp}$ for some $g_{H} \in H,~ g_{H^\perp} \in H^\perp$.
	Therefore, for some $B >0$, we have,
	$$\sum \limits_{l \in \mathfrak L} \sum \limits_{j \in \mathbb{Z}} \sum \limits_{a \in I_{p}}| \langle g, f^{(l)}_{j,a} \rangle |^{2}
	= \sum \limits_{l \in \mathfrak L} \sum \limits_{j \in \mathbb{Z}} \sum \limits_{a \in I_{p}}| \langle g_{H}, f^{(l)}_{j,a} \rangle |^{2} \leq B \|g_{H} \|^{2} \leq B \|g\|^{2}.$$ Thus $\{f_{j,~a}^{(l)} : j \in \mathbb{Z},~ a \in I_{p},~ l \in \mathfrak{L}\}$ is a Besselet with bound $B$.
\end{proof}

By showing a sequence to be a Besselet in a dense subset is sufficient to prove the said sequence is a Besselet in the whole space. Analogous result for frame can be noticed in \cite{7}.

\begin{prop} \label{Besselet}
    If $\mathcal{F}$=$\{f^{(1)}, \ldots, f^{(L)}\}$ is a Besselet for a dense subset $V$ of $L^{2}(\mathbb{Q}_{p})$, then $\mathcal{F}$ is a Besselet for $L^{2}(\mathbb{Q}_{p})$.
\end{prop}

\begin{proof}
Since $\mathcal{F}$=$\{f^{(1)}, \ldots, f^{(L)}\}$ is a Besselet for  $V$, a dense subset of $L^{2}(\mathbb{Q}_{p})$, then there exists a constant $B >0$ and for every $g \in V \subset L^{2}(\mathbb{Q}_{p})$ we have,
 $$\sum \limits_{l \in \mathfrak L} \sum \limits_{j \in \mathbb{Z}} \sum \limits_{a \in I_{p}}| \langle g,f^{(l)}_{j,a} \rangle |^{2} \leq B\left \| g \right \|^{2}.$$
Now for  $g \in L^{2}(\mathbb{Q}_{p})$, suppose $B\left \| g \right \|^{2} < \sum \limits_{l \in \mathfrak L} \sum \limits_{j \in \mathbb{Z}} \sum \limits_{a \in I_{p}}| \langle g,f^{(l)}_{j,a} \rangle |^{2}$. Then there are finite sets $E \subset \mathbb{Z}, ~ F \subset I_{p}$ so that $B\left \| g \right \|^{2} < \sum \limits_{l \in \mathfrak L} \sum \limits_{j \in E} \sum \limits_{a \in F}| \langle g,f^{(l)}_{j,a} \rangle |^{2}$. Again since $V$ is dense in $L^{2}(\mathbb{Q}_{p})$, then there exists $h \in V$ so that $B\left \| h \right \|^{2} < \sum \limits_{l \in \mathfrak L} \sum \limits_{j \in E} \sum \limits_{a \in F}| \langle h,f^{(l)}_{j,a} \rangle |^{2}$, which is a contradiction.
Thus for all $g \in L^{2}(\mathbb{Q}_{p})$ we obtain, $$\sum \limits_{l \in \mathfrak L} \sum \limits_{j \in \mathbb{Z}} \sum \limits_{a \in I_{p}}| \langle g,f^{(l)}_{j,a} \rangle |^{2} \leq
B\left \| g \right \|^{2},$$ and  hence $\mathcal{F}$ is a Besselet in $L^{2}(\mathbb{Q}_{p})$.
\end{proof}

In the following result we extend Lemma \ref{Besselet} for lower framelet condition. The following Lemma shows that every multiframelet on a dense subset is a multiframelet for the whole space.

\begin{prop} \label{dense}
    If $\mathcal{F}$=$\{f^{(1)}, \ldots, f^{(L)}\}$ is a multiframelet for a dense subset $V$ of $L^{2}(\mathbb{Q}_{p})$, then $\mathcal{F}$ is a multiframelet for $L^{2}(\mathbb{Q}_{p})$.
\end{prop}

\begin{proof}
Using definition of $T^*$, equation (\ref{dmf}) can be written as
\begin{equation*} A \|g\|^{2} \leq \|T^{*}g \|^{2} \leq B \|g\|^{2},
~~~ \text{for}~~\text{all} ~ g \in L^{2}(\mathbb{Q}_{p}).
\end{equation*}
Since $\mathcal{F}$ is a multiframelet for  $V$, for every $g \in V$ we have,
\begin{equation*} \label{equ}
A \|g\|^{2} \leq \|T^{*}g \|^{2} \leq B \|g\|^{2}.
\end{equation*}
Therefore, $T^{*}$ is bounded on $V$ and since $V$ is dense in $L^{2}(\mathbb{Q}_{p})$, our assertion is quickly plausible.
\end{proof}

Image of a multiframelet under closed range operator is a multiframelet sequence. An analogous result for frame can be observed in \cite {7}.

\begin{lem}
    Let $\{f_{j,~a}^{(l)} : j \in \mathbb{Z},~ a \in I_{p},~ l \in \mathfrak{L}\}$ be a multiframelet for $L^{2}(\mathbb{Q}_{p})$ with bounds $A,~B$. Then $\{U f_{j,~a}^{(l)} : j \in \mathbb{Z},~ a \in I_{p},~ l \in \mathfrak{L}\}$ is a multiframelet sequence with bounds $A \|U^{\dag}\|^{-2},~ B \|U\|^{2}$, for any closed range bounded, linear operator $U : L^{2}(\mathbb{Q}_{p}) \to L^{2}(\mathbb{Q}_{p})$, where $U^{\dag} : L^{2}(\mathbb{Q}_{p}) \to L^{2}(\mathbb{Q}_{p})$ for which $UU^{\dag}$ is orthogonal projection onto $\mathcal{R}_{U}$.
\end{lem}

\begin{proof}
Let $g \in L^{2}(\mathbb{Q}_{p})$, then we have,
\begin{eqnarray}
\sum \limits_{l \in \mathfrak L} \sum \limits_{j \in \mathbb{Z}} \sum \limits_{a \in I_{p}}| \langle g, Uf^{(l)}_{j,a} \rangle |^{2} &=& \sum \limits_{l \in \mathfrak L} \sum \limits_{j \in \mathbb{Z}} \sum \limits_{a \in I_{p}}| \langle U^{*}g, f^{(l)}_{j,a} \rangle |^{2} \nonumber \\
&\leq&   B \left \| U^{*}g \right\|^{2} \nonumber \\
&\leq& B \left \| U^{*} \right\|^{2} \left \| g \right\|^{2} \nonumber \\
&=& B \left \| U \right\|^{2} \left \| g \right\|^{2}, \nonumber
\end{eqnarray}
which proves that $\{U f_{j,~a}^{(l)} : j \in \mathbb{Z},~ a \in I_{p},~ l \in \mathfrak{L}\}$ is a Besselet with bound $B \| U \|^{2}$. 
Let $h \in span \{ U f_{j,~a}^{(l)} : j \in \mathbb{Z},~ a \in I_{p},~ l \in \mathfrak{L} \}$. Then $h = Uf$ for some $f \in span \{ f_{j,~a}^{(l)} : j \in \mathbb{Z},~ a \in I_{p},~ l \in
\mathfrak{L} \}$. By Lemma A.7.2 of \cite{7}, there is a bounded operator $U^{\dag} : L^{2}(\mathbb{Q}_{p}) \to L^{2}(\mathbb{Q}_{p})$ such that $UU^{\dag}$ is orthogonal projection onto $\mathcal{R}_{U}$ and hence is self-adjoint.
Therefore we obtain,
$$ h = Uf = (UU^{\dag})Uf = (UU^{\dag})^{*}Uf = (U^{\dag})^{*}U^{*}Uf.$$
\begin{eqnarray} \text{Thus   } \|h\|^{2} \leq \|(U^{\dag})^{*} \|^{2} \| U^{*}Uf\|^{2} &\leq& \frac{\|(U^{\dag})^{*} \|^{2}}{A} \sum \limits_{l \in \mathfrak L} \sum \limits_{j \in \mathbb{Z}} \sum \limits_{a \in I_{p}}| \langle U^{*}Uf, f^{(l)}_{j,a} \rangle |^{2} \nonumber \\ &=& \frac{\|(U^{\dag})^{*} \|^{2}}{A} \sum \limits_{l \in \mathfrak L} \sum \limits_{j \in \mathbb{Z}} \sum \limits_{a \in I_{p}}| \langle h, Uf^{(l)}_{j,a} \rangle |^{2} \nonumber
\end{eqnarray}
\begin{equation} \label{liq} \text{and}~~\text{therefore}~~ \frac{A}{\|(U^{\dag})^{*} \|^{2}}\|h\|^{2} \leq \sum \limits_{l \in \mathfrak L} \sum \limits_{j \in \mathbb{Z}} \sum \limits_{a \in I_{p}}| \langle h, Uf^{(l)}_{j,a} \rangle |^{2}.
\end{equation}
So lower multiframelet condition satisfies for every $h \in span \{ Uf_{j,~a}^{(l)} : j \in \mathbb{Z},~ a \in I_{p},~ l \in \mathfrak{L}\}$ and hence using Proposition \ref{dense}, our assertion is tenable.
\end{proof}

In the following result we discuss a necessary and sufficient condition of a multiframelet sequence to be a multiframelet. Similar result for frame can be observed in \cite{7}.

\begin{lem}
    If $\{f_{j,~a}^{(l)} : j \in \mathbb{Z},~ a \in I_{p},~ l \in \mathfrak{L} \}$ is a multiframelet sequence in $L^{2}(\mathbb{Q}_{p})$ with the associated synthesis operator $T$, then $\{f_{j,~a}^{(l)} : j \in \mathbb{Z},~ a \in I_{p},~ l \in \mathfrak{L} \}$ is multiframelet for $L^{2}(\mathbb{Q}_{p})$ if and only if $T^*$ is injective.
\end{lem}

\begin{proof}
Since $\{f_{j,~a}^{(l)} : j \in \mathbb{Z},~ a \in I_{p},~ l \in \mathfrak{L} \}$ is a multiframelet sequence, we have,
\begin{equation*}
\mathcal{R}_{T} = \overline{span} \{f_{j,a}^{(l)}~:~j \in
\mathbb{Z},~ a \in I_{p},~ l \in \mathfrak{L} \}.
\end{equation*}
Again it is well-known that $\mathcal{N}_{T^*} = \mathcal{R}_{T}^{\perp}$, $\mathcal{R}_{T} \bigoplus \mathcal{R}_{T}^{\perp} = L^{2}(\mathbb{Q}_{p})$ and therefore we have,
\begin{equation*}
T^{*} \text{ is injective} \Leftrightarrow \mathcal{N}_{T^*} = \{0\}
\Leftrightarrow \mathcal{R}_{T}^{\perp} = \{0\} \Leftrightarrow
\mathcal{R}_{T}= L^{2}(\mathbb{Q}_{p}) .
\end{equation*}
 Thus $\{f_{j,~a}^{(l)} : j \in \mathbb{Z},~ a \in I_{p},~ l \in \mathfrak{L} \}$ is a multiframelet
for $L^{2}(\mathbb{Q}_{p})$. 

Conversely, if $\{f_{j,~a}^{(l)} : j \in \mathbb{Z},~ a \in I_{p},~ l \in \mathfrak{L} \}$ is a multiframelet
for $L^{2}(\mathbb{Q}_{p})$, then  is easily followed that $ T^{*}$ is injective.
\end{proof}

The following Theorem presents a necessary and sufficient condition of a sequence to be a multiframelet. An analogous result for frame was studied by Aldroubi  in \cite{13}. 

\begin{thm}
    Let $\{f_{j,~a}^{(l)} : j \in \mathbb{Z},~ a \in I_{p},~ l \in \mathfrak{L}\}$ be a multiframelet with bounds $A,~B$ and suppose $\mathcal{M} : L^{2}(\mathbb{Q}_{p}) \to L^{2}(\mathbb{Q}_{p})$ is a bounded linear operator. Then   $\{\mathcal{M}f_{j,~a}^{(l)} : j \in \mathbb{Z},~ a \in I_{p},~ l \in \mathfrak{L}\}$ is a multiframelet if and only if  for every $g \in L^{2}(\mathbb{Q}_{p})$, there exists $\lambda >0$ so that $\lambda \|g\|^{2} \leq \|\mathcal{M}^{*}g\|^{2}$.
    
\end{thm}

\begin{proof}
Let $g \in L^{2}(\mathbb{Q}_{p})$ and suppose $\{\mathcal{M}f_{j,~a}^{(l)} : j \in \mathbb{Z},~ a \in I_{p},~ l \in \mathfrak{L}\}$ is a multiframelet with lower bound $C$. Then for every $g \in L^{2}(\mathbb{Q}_{p})$ we get,
 $$C \|g\|^{2} \leq \sum \limits_{l \in \mathfrak L} \sum \limits_{j \in \mathbb{Z}} \sum \limits_{a \in I_{p}}| \langle g, \mathcal{M}f^{(l)}_{j,a} \rangle |^{2} = \sum \limits_{l \in \mathfrak L} \sum \limits_{j \in \mathbb{Z}} \sum \limits_{a \in I_{p}}| \langle \mathcal{M}^{*}g, f^{(l)}_{j,a} \rangle |^{2} \leq B \|\mathcal{M}^{*}g \|^{2}$$ and hence our assertion is tenable by choosing $\lambda = \frac{C}{B}$.

Conversely, suppose for every $g \in L^{2}(\mathbb{Q}_{p})$, there exists $\lambda >0$ so that $\lambda \|g\|^{2} \leq \|\mathcal{M}^{*}g\|^{2}$. Then we obtain,
 $$ A \lambda \|g\|^{2} \leq A
\|\mathcal{M}^{*}g \|^{2} \leq \sum \limits_{l \in \mathfrak L} \sum
\limits_{j \in \mathbb{Z}} \sum \limits_{a \in I_{p}}| \langle
\mathcal{M}^{*}g, f^{(l)}_{j,a} \rangle |^{2} = \sum
\limits_{l \in \mathfrak L} \sum \limits_{j \in \mathbb{Z}} \sum \limits_{a
\in I_{p}}| \langle g, \mathcal{M}f^{(l)}_{j,a} \rangle |^{2} .$$
Similarly, using upper multiframelet condition we get,
 $$\sum \limits_{l \in \mathfrak L} \sum \limits_{j \in \mathbb{Z}} \sum \limits_{a \in I_{p}}| \langle g,
\mathcal{M}f^{(l)}_{j,a} \rangle |^{2} = \sum \limits_{l \in \mathfrak L} \sum
\limits_{j \in \mathbb{Z}} \sum \limits_{a \in I_{p}}| \langle
\mathcal{M}^{*}g, f^{(l)}_{j,a} \rangle |^{2} \leq B
\|\mathcal{M}^{*}g \|^{2} \leq B \|\mathcal{M}^{*}\|^{2} \|g
\|^{2}.$$ 
Therefore, $\{\mathcal{M}f_{j,~a}^{(l)} : j \in \mathbb{Z},~ a \in
I_{p},~ l \in \mathfrak{L}\}$ is a multiframelet with bounds
$A\lambda,~ B\|\mathcal{M}^{*}\|^{2}$.
\end{proof}

In the following two results we characterize multiframelets by means of erasure and perturbation.

\begin{thm}
    Let $\{f_{j,~a}^{(l)} : j \in \mathbb{Z},~ a \in I_{p},~ l \in \mathfrak{L}\}$ be a multiframelet with bounds $A,~B$. Suppose $I_q \subset I_p$ so that $\{f_{j,~a}^{(l)} : j \in \mathbb{Z},~ a \in I_{q},~ l \in \mathfrak{L}\}$ is a Besselet with bound $C$, where $C < A$. Then $\{f_{j,~a}^{(l)} : j \in \mathbb{Z},~ a \in {I_{p} \setminus I_q},~ l \in \mathfrak{L}\}$ is a multiframelet with bounds $(A-C)$ and $B$.
    
\end{thm}

\begin{proof}
Since $\{f_{j,~a}^{(l)} : j \in \mathbb{Z},~ a \in I_{p},~ l \in \mathfrak{L}\}$ is a multiframelet with bounds $A, B$ and  $\{f_{j,~a}^{(l)} : j \in \mathbb{Z},~ a \in I_{q},~ l \in \mathfrak{L}\}$ is a Besselet with bound $C$, for every $g \in L^{2}(\mathbb{Q}_{p})$ we get,

\begin{eqnarray*}
\sum \limits_{l \in \mathfrak L} \sum \limits_{j \in \mathbb{Z}} \sum \limits_{a \in {I_p \setminus I_q}} | \langle g, f^{(l)}_{j,a} \rangle |^{2} &=& \sum \limits_{l \in \mathfrak L} \sum \limits_{j \in \mathbb{Z}} \sum \limits_{a \in {I_p }} | \langle g, f^{(l)}_{j,a} \rangle |^{2} - \sum \limits_{l \in \mathfrak L} \sum \limits_{j \in \mathbb{Z}} \sum \limits_{a \in {I_q}} | \langle g, f^{(l)}_{j,a} \rangle |^{2} 
\\ & \geq & (A-C)\|g\|^2.
\end{eqnarray*}
and $\sum \limits_{l \in \mathfrak L} \sum \limits_{j \in \mathbb{Z}} \sum \limits_{a \in {I_p \setminus I_q}} | \langle g, f^{(l)}_{j,a} \rangle |^{2} \leq \sum \limits_{l \in \mathfrak L} \sum \limits_{j \in \mathbb{Z}} \sum \limits_{a \in {I_p }} | \langle g, f^{(l)}_{j,a} \rangle |^{2} \leq B\|g\|^2.$
\end{proof}

\begin{thm}
    Let $\{f_{j,~a}^{(l)} : j \in \mathbb{Z},~ a \in I_{p},~ l \in \mathfrak{L}\}$ be a multiframelet with bounds $A,~B$. Suppose $\{g_{j,~a}^{(l)} : j \in \mathbb{Z},~ a \in I_{p},~ l \in \mathfrak{L}\}$ is another collection so that for some $0 <C<A$ and every $g \in L^{2}(\mathbb{Q}_{p})$  we have,
\begin{equation} \label{eqn}
\sum \limits_{l \in \mathfrak L} \sum \limits_{j \in \mathbb{Z}} \sum \limits_{a \in {I_p }} | \langle g, (f^{(l)}_{j,a} - g^{(l)}_{j,a})\rangle |^{2} \leq C\|g\|^2.
\end{equation}
Then $\{g_{j,~a}^{(l)} : j \in \mathbb{Z},~ a \in I_{p},~ l \in \mathfrak{L}\}$ forms a multiframelet with bounds $(\sqrt A - \sqrt C)^2, (\sqrt B + \sqrt C)^2$.  
\end{thm}

\begin{proof}
Using equation (\ref{eqn}) for every $g \in L^{2}(\mathbb{Q}_{p})$ we obtain,
\begin{eqnarray*}
\left (\sum \limits_{l \in \mathfrak L} \sum \limits_{j \in \mathbb{Z}} \sum \limits_{a \in {I_p }} | \langle g, g^{(l)}_{j,a} \rangle |^{2} \right )^{\frac {1} {2}} &\leq & \left (\sum \limits_{l \in \mathfrak L} \sum \limits_{j \in \mathbb{Z}} \sum \limits_{a \in {I_p }} | \langle g, (f^{(l)}_{j,a} - g^{(l)}_{j,a} ) \rangle |^{2}   \right )^{\frac {1} {2}} 
\\ &+& \left (\sum \limits_{l \in \mathfrak L} \sum \limits_{j \in \mathbb{Z}} \sum \limits_{a \in {I_p }} | \langle g, f^{(l)}_{j,a} \rangle |^{2} \right )^{\frac {1} {2}}
\\ & \leq &  (\sqrt B+ \sqrt C)\|g\|.
\end{eqnarray*}
Therefore, $\sum \limits_{l \in \mathfrak L} \sum \limits_{j \in \mathbb{Z}} \sum \limits_{a \in {I_p }} | \langle g, g^{(l)}_{j,a} \rangle |^{2} \leq (\sqrt B + \sqrt C)^2 \|g\|^2$.

Analogously using equation (\ref{eqn}) for every $g \in L^{2}(\mathbb{Q}_{p})$ we get, $$\sum \limits_{l \in \mathfrak L} \sum \limits_{j \in \mathbb{Z}} \sum \limits_{a \in {I_p }} | \langle g, g^{(l)}_{j,a} \rangle |^{2} \geq (\sqrt A - \sqrt C)^2 \|g\|^2.$$
Hence our assertion is tenable.
\end{proof}

Duffin and Schaeffer studied properties of frame operator on $\mathbb R$ in \cite{7, 1}. Later, Heil continued to study the same in general Hilbert space (cf. \cite{11}) and Debnath independently studied this in general Hilbert space (cf. \cite{14}). In the rest of this section we produce various characterizations of multiframelet by means of associated multiframelet operator $\mathcal{S}$.

\begin{prop}
    If $\{f^{(l)}_{j,a} : j \in \mathbb{Z},~ a \in I_{p},~ l \in \mathfrak{L} \}$ is a multiframelet with bounds $A, B$ and the associated multiframelet operator $\mathcal{S}$, then $\{\mathcal{S}^{-1}f^{(l)}_{j,a} : j \in \mathbb{Z},~ a \in I_{p},~ l \in \mathfrak{L} \}$ is a multiframelet with bounds $B^{-1} ,~ A^{-1}$. Furthermore, if $A,~B$ are optimal bounds for $\{ f^{(l)}_{j,a} : j \in \mathbb{Z},~ a \in I_{p},~ l \in \mathfrak{L} \}$ then $B^{-1} ,~ A^{-1}$ are optimal bounds for $\{\mathcal{S}^{-1}f^{(l)}_{j,a} : j \in \mathbb{Z},~ a \in I_{p},~ l \in \mathfrak{L} \}$ whose multiframelet operator is $\mathcal{S}^{-1}$.
\end{prop}

\begin{proof} Since $\{f^{(l)}_{j,a} : j \in \mathbb{Z},~ a \in I_{p},~ l \in \mathfrak{L} \}$ is a multiframelet with bounds $A, B$ and the associated multiframelet operator $\mathcal{S}$ we have, 
\begin{equation} \label{mul_operator} A \mathcal{I} \leq \mathcal{S} \leq B \mathcal{I} \Rightarrow 0 \leq \mathcal{I} - B^{-1}\mathcal{S} \leq \frac{B-A}{A} \mathcal{I}.
\end{equation}
Hence $\|\mathcal{I} - B^{-1}\mathcal{S} \| < 1$ and consequently  $\mathcal{S}$ is invertible. Therefore, for every $g \in L^{2}(\mathbb{Q}_{p})$ we obtain,
$$\sum \limits_{l \in \mathfrak L} \sum \limits_{j \in \mathbb{Z}} \sum \limits_{a \in I_{p}} |\langle g, \mathcal{S}^{-1}f^{(l)}_{j,a} \rangle|^{2} = \sum \limits_{l \in \mathfrak L} \sum \limits_{j \in \mathbb{Z}} \sum \limits_{a \in I_{p}} |\langle \mathcal{S}^{-1}g, f^{(l)}_{j,a} \rangle|^{2} \leq B \|\mathcal{S}^{-1}g\|^{2} \leq B \|\mathcal{S}^{-1}\|^{2} \|g\|^{2}.$$ 
Thus $\{\mathcal{S}^{-1}f^{(l)}_{j,a} : j \in \mathbb{Z},~ a \in I_{p},~ l \in \mathfrak{L} \}$ is a Besselet and the corresponding  multiframelet operator for $\{\mathcal{S}^{-1}f^{(l)}_{j,a} : j \in \mathbb{Z},~ a \in I_{p},~ l \in \mathfrak{L} \}$ is well-defined.
For every $g \in L^{2}(\mathbb{Q}_{p})$ we  have,
\begin{eqnarray*} \sum \limits_{l \in \mathfrak L} \sum \limits_{j \in \mathbb{Z}} \sum
\limits_{a \in I_{p}} \langle g, \mathcal{S}^{-1}f^{(l)}_{j,a} \rangle \mathcal{S}^{-1}f^{(l)}_{j,a} &=& \mathcal{S}^{-1} \sum \limits_{l \in \mathfrak L} \sum \limits_{j \in \mathbb{Z}} \sum \limits_{a
\in I_{p}} \langle g, \mathcal{S}^{-1}f^{(l)}_{j,a} \rangle f^{(l)}_{j,a}
\\ &= & \mathcal{S}^{-1} \sum \limits_{l \in \mathfrak L} \sum \limits_{j \in \mathbb{Z}} \sum \limits_{a \in I_{p}} \langle \mathcal{S}^{-1}g, f^{(l)}_{j,a} \rangle f^{(l)}_{j,a}
\\ & = & \mathcal{S}^{-1}\mathcal{S}\mathcal{S}^{-1}g
\\ & = & \mathcal{S}^{-1}g.
\end{eqnarray*}
Therefore, the multiframelet operator for $\{\mathcal{S}^{-1}f^{(l)}_{j,a} : j \in \mathbb{Z},~ a \in I_{p},~ l \in \mathfrak{L} \}$ is $\mathcal{S}^{-1}$. Thus applying equation (\ref{mul_operator}), for every $g \in L^{2}(\mathbb{Q}_{p})$ we obtain,
 $B^{-1}\|g\|^{2} \leq \langle \mathcal{S}^{-1}g,g \rangle \leq A^{-1} \|g\|^{2}$ and therefore $\{\mathcal{S}^{-1}f^{(l)}_{j,a} : j \in \mathbb{Z},~ a \in I_{p},~ l \in \mathfrak{L} \}$ is a multiframelet with bounds $B^{-1}, A^{-1}$.

In order to prove the optimality of bounds, let $B$ be the optimal upper bound for the multiframelet $\{f^{(l)}_{j,a} : j \in \mathbb{Z},~ a \in I_{p},~ l \in \mathfrak{L} \}$ and assume that the optimal
lower bound for $\{\mathcal{S}^{-1}f^{(l)}_{j,a} : j \in \mathbb{Z},~ a \in I_{p},~ l \in \mathfrak{L} \}$ is $D > B^{-1}$. Then $\{f^{(l)}_{j,a} : j \in \mathbb{Z},~ a \in I_{p},~ l \in \mathfrak{L} \}=
\{(\mathcal{S}^{-1})^{-1}\mathcal{S}^{-1}f^{(l)}_{j,a} : j \in \mathbb{Z},~ a \in I_{p},~ l \in \mathfrak{L} \}$ has the upper bound $D^{-1}<B$, which is a contradiction. Hence
$\{\mathcal{S}^{-1}f^{(l)}_{j,a} : j \in \mathbb{Z},~ a \in I_{p},~l \in \mathfrak{L} \}$ has the optimal lower bound $B^{-1}$. By applying similar argument, the optimal upper bound will be achieved.
\end{proof}

$\{\mathcal{S}^{-1}f^{(l)}_{j,a} : j \in \mathbb{Z},~ a \in I_{p},~ l \in \mathfrak{L} \}$ is called the \textbf{canonical dual multiframelet} of $\{f^{(l)}_{j,a} : j \in \mathbb{Z},~ a \in I_{p},~ l \in \mathfrak{L} \}$. Note that
$\{\mathcal{S}f^{(l)}_{j,a} : j \in \mathbb{Z},~ a \in I_{p},~ l \in \mathfrak{L} \}$ is also a multiframelet.


\begin{thm} \label{decom} (Multiframelet decomposition)
    Let $\mathcal {F}$=$\{f^{(1)}, \ldots, f^{(L)}\}$ be a multiframelet for $L^{2}(\mathbb{Q}_{p})$ with the corresponding multiframelet operator $\mathcal{S}$. Then every  $g \in  L^{2}(\mathbb{Q}_{p})$ has the following representation
    \begin{equation*}
    g = \sum \limits_{l \in \mathfrak L} \sum \limits_{j \in \mathbb{Z}} \sum \limits_{a \in I_{p}} \langle g , \mathcal{S}^{-1} f_{j,a}^{(l)} \rangle f_{j,a}^{(l)} = \sum \limits_{l \in \mathfrak L} \sum \limits_{j \in \mathbb{Z}} \sum \limits_{a \in I_{p}} \langle g, f_{j,a}^{(l)} \rangle \mathcal{S}^{-1} f_{j,a}^{(l)}, 
    \end{equation*}
the above sums converges unconditionally.
\end{thm}

\begin{proof} Let $g \in L^{2}(\mathbb{Q}_{p})$, the we have,
\begin{equation*}
g = \mathcal{S}\mathcal{S}^{-1}g = \sum \limits_{l \in \mathfrak L} \sum
\limits_{j \in \mathbb{Z}} \sum \limits_{a \in I_{p}} \langle
\mathcal{S}^{-1}g,  f_{j,a}^{(l)} \rangle f_{j,a}^{(l)} = \sum
\limits_{l \in \mathfrak L} \sum \limits_{j \in \mathbb{Z}} \sum \limits_{a
\in I_{p}} \langle g, \mathcal{S}^{-1} f_{j,a}^{(l)} \rangle
f_{j,a}^{(l)}.
\end{equation*}
Similarly, we have following representation
\begin{eqnarray}
g = \mathcal{S}^{-1}\mathcal{S}g &=& \mathcal{S}^{-1}\sum
\limits_{l \in \mathfrak L} \sum \limits_{j \in \mathbb{Z}} \sum \limits_{a
\in I_{p}} \langle g ,  f_{j,a}^{(l)} \rangle f_{j,a}^{(l)}
\nonumber = \sum \limits_{l \in \mathfrak L} \sum \limits_{j \in
\mathbb{Z}} \sum \limits_{a \in I_{p}} \langle g , f_{j,a}^{(l)}
\rangle \mathcal{S}^{-1} f_{j,a}^{(l)}
\end{eqnarray}
The unconditionally convergence follows from the fact that both $\{f^{(l)}_{j,a} : j \in \mathbb{Z},~ a \in I_{p},~ l \in \mathfrak{L} \}$ and $\{\mathcal{S}^{-1}f^{(l)}_{j,a} : j \in \mathbb{Z},~ a \in I_{p},~ l \in \mathfrak{L} \}$ are multiframelet.
\end{proof}


\begin{rem}
    Theorem (\ref{decom}) also proves that $\mathcal{S}$ is surjective and therefore a topological isomorphism of $L^{2}(\mathbb{Q}_{p})$. If $\mathcal {F}$ is a tight multiframelet, then $\mathcal{S}^{-1}=A^{-1} \mathcal{I}$, which leads to another representation of $g$ as $g = A^{-1} \sum \limits_{l \in \mathfrak L} \sum \limits_{j \in \mathbb{Z}} \sum \limits_{a \in I_{p}} \langle g , f_{j,a}^{(l)} \rangle f_{j,a}^{(l)}$. Note that $A$ is an eigen value of $\mathcal{S}$. Also by multiframelet decomposition theorem, it can be concluded that a multiframelet can be represented by other multiframelet in $L^{2}(\mathbb{Q}_{p})$.
\end{rem}

\begin{thm}
    Let $\{f_{j,~a}^{(l)} : j \in \mathbb{Z},~ a \in I_{p},~ l \in \mathfrak{L} \}$ and $\{g_{j,~a}^{(l)} : j \in \mathbb{Z},~ a \in I_{p},~ l \in \mathfrak{L} \}$ be Besselets in $L^{2}(\mathbb{Q}_{p})$. Then for every $f, g \in L^{2}(\mathbb{Q}_{p})$ the following statements are equivalent :
    \begin{itemize}
        \item[(i)]  $f= \sum \limits_{l \in \mathfrak L} \sum \limits_{j \in \mathbb{Z}} \sum \limits_{a \in I_{p}} \langle f, f_{j,a}^{(l)} \rangle  g_{j,a}^{(l)}$.
        \item[(ii)] $f= \sum \limits_{l \in \mathfrak L} \sum \limits_{j \in \mathbb{Z}} \sum \limits_{a \in I_{p}} \langle f , g_{j,a}^{(l)} \rangle  f_{j,a}^{(l)}$.
        \item[(iii)] $\langle f,g \rangle = \sum \limits_{l \in \mathfrak L} \sum \limits_{j \in \mathbb{Z}} \sum \limits_{a \in I_{p}} \langle f ,  g_{j,a}^{(l)} \rangle  \langle f_{j,a}^{(l)}, g \rangle$.
    \end{itemize}
    Furthermore, if one of the above equivalent conditions is satisfied, then  $\{f_{j,~a}^{(l)} : j \in \mathbb{Z},~ a \in I_{p},~ l \in \mathfrak{L} \}$ and $\{g_{j,~a}^{(l)} : j \in \mathbb{Z},~ a \in I_{p},~ l \in \mathfrak{L} \}$ are dual multiframelets for $L^{2}(\mathbb{Q}_{p})$.
\end{thm}

\begin{proof}
Let $U$ and $T$ be synthesis operators corresponding to $\{f_{j,~a}^{(l)} : j \in \mathbb{Z},~ a \in I_{p},~ l \in \mathfrak{L} \}$ and $\{g_{j,~a}^{(l)} : j \in \mathbb{Z},~ a \in I_{p},~ l \in \mathfrak{L} \}$ respectively.\\

$\underline{(i) \Longleftrightarrow (ii)}$  $(i)$ is equivalent to, $TU^* = \mathcal{I} \Leftrightarrow (TU^*)^* = \mathcal{I}^* \Leftrightarrow UT^* = \mathcal{I}$, which is equivalent to $(ii)$.\\ \\
$\underline{(ii) \Longrightarrow (iii)}$ Let $f,g \in
L^{2}(\mathbb{Q}_{p})$. Now
$$\langle f,g \rangle = \left \langle \sum \limits_{l \in \mathfrak L} \sum \limits_{j \in \mathbb{Z}} \sum \limits_{a \in I_{p}} \langle f , g_{j,a}^{(l)} \rangle   f_{j,a}^{(l)}, ~ g \right \rangle = \sum \limits_{l \in \mathfrak L} \sum \limits_{j \in \mathbb{Z}} \sum \limits_{a \in I_{p}} \langle f , g_{j,a}^{(l)} \rangle  \langle f_{j,a}^{(l)},~ g \rangle.$$\\
$\underline{(iii) \Longrightarrow (ii)}$
Let $f \in L^{2}(\mathbb{Q}_{p})$. Then $\sum \limits_{l \in \mathfrak L} \sum \limits_{j \in \mathbb{Z}} \sum \limits_{a \in I_{p}} \langle f ~,~ g_{j,a}^{(l)} \rangle  f_{j,a}^{(l)}$ is well-defined  in $L^{2}(\mathbb{Q}_{p})$. Therefore for every $g \in L^{2}(\mathbb{Q}_{p})$ we obtain,
$\langle f,g \rangle = \sum \limits_{l \in \mathfrak L} \sum \limits_{j \in \mathbb{Z}} \sum \limits_{a \in I_{p}} \langle f , g_{j,a}^{(l)} \rangle  \langle f_{j,a}^{(l)},  g \rangle$ and hence
$ \langle f - \sum \limits_{l \in \mathfrak L} \sum \limits_{j \in \mathbb{Z}} \sum \limits_{a \in I_{p}} \langle f , g_{j,a}^{(l)} \rangle f_{j,a}^{(l)},  g \rangle =0$ and consequently, $ f = \sum \limits_{l \in \mathfrak L} \sum \limits_{j \in \mathbb{Z}} \sum \limits_{a \in I_{p}} \langle f , g_{j,a}^{(l)} \rangle f_{j,a}^{(l)}$.
Furthermore, if one of the above equivalent conditions is satisfied, then considering
$\{g_{j,~a}^{(l)} : j \in \mathbb{Z},~ a \in I_{p},~ l \in
\mathfrak{L} \}$ is Besselet with bound $B$, we obtain,
\begin{eqnarray}
\|f\|^{2} &=& \sum \limits_{l \in \mathfrak L} \sum \limits_{j \in \mathbb{Z}}
\sum \limits_{a \in I_{p}} \langle f , g_{j,a}^{(l)} \rangle
\langle f_{j,a}^{(l)}~,~ f \rangle \nonumber \\ &\leq& \left(\sum
\limits_{l \in \mathfrak L} \sum \limits_{j \in \mathbb{Z}} \sum \limits_{a
\in I_{p}} |\langle f , g_{j,a}^{(l)} \rangle |^{2}  \right
)^\frac{1}{2}  \left(\sum \limits_{l \in \mathfrak L} \sum \limits_{j \in
\mathbb{Z}} \sum \limits_{a \in I_{p}} |\langle f_{j,a}^{(l)}, f
\rangle |^{2}  \right )^\frac{1}{2} (\text{by cauchy-schwarz })
\nonumber \\  &\leq& \sqrt{B} ~ \|f\|~ \left( \sum \limits_{l \in \mathfrak L}
\sum \limits_{j \in \mathbb{Z}} \sum \limits_{a \in I_{p}} |\langle
f_{j,a}^{(l)}, f \rangle |^{2}  \right )^\frac{1}{2}  \nonumber
\end{eqnarray}
and hence $\frac{1}{B} \|f\|^{2} \leq \sum \limits_{l \in \mathfrak L} \sum
\limits_{j \in \mathbb{Z}} \sum \limits_{a \in I_{p}} |\langle
f_{j,a}^{(l)}, f \rangle |^{2}$. Therefore, $\{f_{j,~a}^{(l)} : j \in
\mathbb{Z},~ a \in I_{p},~ l \in \mathfrak{L} \}$ is a multiframelet.
\end{proof}

\begin{thm}
    If $\{f_{j,~a}^{(l)} : j \in \mathbb{Z},~ a \in I_{p},~ l \in \mathfrak{L} \}$ is a multiframelet for $L^{2}(\mathbb{Q}_{p})$, then $\{f_{j,~a}^{(l)} : j \in \mathbb{Z},~ a \in I_{p},~ l \in \mathfrak{L} \}$ is also a tight multiframelet if and only if  for some $\alpha > 0$, $\{\alpha f_{j,~a}^{(l)} : j \in \mathbb{Z},~ a \in I_{p},~ l \in \mathfrak{L} \}$ is a dual of $\{f_{j,~a}^{(l)} : j \in \mathbb{Z},~ a \in I_{p},~ l \in \mathfrak{L} \}$.
\end{thm}

\begin{proof}
Let $\{f_{j,~a}^{(l)} : j \in \mathbb{Z},~ a \in I_{p},~ l \in \mathfrak{L} \}$ be a tight multiframelet with bound $\frac{1}{\alpha }$, with the associated  multiframelet operator $\mathcal{S}$.
Then applying multiframelet decomposition theorem, for every $ f \in L^{2}(\mathbb{Q}_{p})$ we have,
\begin{eqnarray}
f = \sum \limits_{l \in \mathfrak L} \sum \limits_{j \in \mathbb{Z}} \sum
\limits_{a \in I_{p}} \langle f, \mathcal{S}^{-1} f_{j,a}^{(l)}
\rangle f_{j,a}^{(l)} \nonumber = \sum \limits_{l \in \mathfrak L} \sum
\limits_{j \in \mathbb{Z}} \sum \limits_{a \in I_{p}} \langle f, \alpha 
f_{j,a}^{(l)} \rangle f_{j,a}^{(l)}. \nonumber
\end{eqnarray}
Hence $\{\alpha  f_{j,~a}^{(l)} : j \in \mathbb{Z},~ a \in I_{p}, l \in \mathfrak{L} \}$ is a dual of $\{f_{j,~a}^{(l)} : j \in \mathbb{Z},~ a \in I_{p},~ l \in \mathfrak{L} \}$.

Conversely, suppose for some $\alpha >0$, $\{\alpha  f_{j,~a}^{(l)} : j \in \mathbb{Z},~
a \in I_{p},~ l \in \mathfrak{L} \}$  is a dual of $\{f_{j,~a}^{(l)} : j \in \mathbb{Z},~ a \in I_{p},~ l \in \mathfrak{L} \}$. Then for every $ g \in L^{2}(\mathbb{Q}_{p})$ we obtain,
$$\|g\|^{2} = \sum \limits_{l \in \mathfrak L} \sum \limits_{j \in \mathbb{Z}} \sum \limits_{a \in I_{p}} \langle g, \alpha  f_{j,a}^{(l)} \rangle  \langle f_{j,a}^{(l)}, g \rangle = \alpha  \sum \limits_{l \in \mathfrak L} \sum \limits_{j \in \mathbb{Z}} \sum \limits_{a \in I_{p}} |\langle g , f_{j,a}^{(l)} \rangle |^{2}.$$ 
Thus $\{f_{j,~a}^{(l)} : j \in \mathbb{Z},~ a \in I_{p},~ l \in \mathfrak{L} \}$ is a tight multiframelet for $L^{2}(\mathbb{Q}_{p})$ with bound $\frac{1}{\alpha }$.
\end{proof}

\section*{Acknowledgment}
The authors acknowledge the financial support of the Ministry of Human Resource Development (M.H.R.D.), Government of India. Furthermore, they would like to express their sincere gratitude to Dr. Divya Singh and Dr. Saikat Mukherjee for their guidance to prepare this article.



\end{document}